\documentclass[a4paper,12pt]{amsart}

\usepackage{epsfig}
\usepackage{amsthm,amsfonts}
\usepackage{amssymb,graphicx,color}
\usepackage[all]{xy}
\usepackage{verbatim}
\usepackage{hyperref}
\usepackage{enumitem}

\usepackage[a4paper,top=2.5cm,bottom=2.5cm,left=3cm,right=3cm]{geometry}

\newtheorem{theorem}{Theorem}[section]
\newtheorem*{theorem*}{Theorem}

\newcommand{\C}{\mathbb{C}}

\title[On finite polynomial mappings ]{On finite polynomial mappings}

\author{Zbigniew Jelonek}
\address{Zbigniew Jelonek -
Instytut Matematyczny, Polska Akademia Nauk, \'Sniadeckich 8, 00-956 Warszawa, Poland.}
\email{najelone@cyf-kr.edu.pl}

\subjclass{14 D 05, 14 R 99.}

\begin{document}

\begin{abstract}
Let $X\subset \C^n$ be a smooth irreducible  affine variety of dimension $k$ and let $F: X\to \C^m$ be a polynomial mapping.
We prove that if $m\ge k$, then there is a Zariski open dense subset $U$ in the space of linear mappings ${\mathcal L}(\C^n,\C^m)$ 
such that for every $L\in U$ the mapping $F+L$ is a finite mapping. Moreover,  we can choose $U$ in this way, that all  mappings $F+L; L\in U$ are topologically equivalent.
\end{abstract}

\maketitle

 \section{Introduction}
Assume that we have an algebraic family $\mathcal F$ of polynomial generically-finite mappings $f_p: X\to \C^m; \ p\in {\mathcal F}$, where
$X$ is a smooth irreducible affine variety. It is important to know the behavior of proper mappings in a such family. In general, proper mappings does not form an algebraic subset of $\mathcal F$, but only constructible one.  However we show in this note that we have some regular behavior in such  family.

 As an application we show  that if $X\subset \C^n$ is a smooth irreducible affine variety of dimension $k$ and let $F: X\to \C^m$ be a polynomial mapping.
If $m\ge k$, then there exists a Zariski open dense subset $U$ in the space of linear mappings ${\mathcal L}(\C^n,\C^m)$ 
such that:

a) for every $L\in U$ the mapping $F+L$ is a finite mapping. 

b) all mappings $F+L, L\in U$ are topologically equivalent.

Let us recall that  mappings $f,g: X\to Y$ are topologically equivalent, if there exist homeomorphisms $\phi: X\to X$ and $\psi: Y\to Y$ such that $f=\psi\circ g\circ \phi.$

\medskip

\section{Main results}
Let us start with the following:

\begin{theorem}\label{rodzina}
Let $P, X, Y$ be  smooth irreducible  affine algebraic varieties and let $F: P\times X\to P\times Y$ be a generically finite mapping.
The mapping $F$ induces a family ${\mathcal F}=\{ f_p(\cdot)=F(p,\cdot), \ p\in P \}.$ Then either there exists a Zariski open dense subset 
$U\subset P$ such that for every $p\in P$  a  mapping $f_p$ is proper,
or there exists a Zariski open dense subset 
$V\subset P$ such that for every $p\in P$ a mapping  $f_p$ is not proper.

In the first case we have:

a) foe every  non-proper mappings $f_p$ in the family $\mathcal F$ we have  $\mu(f_p)<\mu(F)$, where $\mu(f)$ denotes the geometric degree of $f,$

b) all generic mappings $f_p$ are topologically equivalent, i.e., there exists a Zariski open dense subset $W\subset P$, such that for every $p,q\in W$ mappings $f_p$ and $f_q$ are topologically equivalent.
\end{theorem}

\begin{proof}
First of all note that for every $(p,x)\in P\times X$ we have $\mu_{(p,x)}(F)= \mu_x (f_p)$ (here $\mu_x(f)$ denotes the local multiplicity of $f$ in $x$). In the sequel we use the fact that a mapping $g: X\to Y$ is proper
over a point $y\in Y$ if and only if $\sum_{g(x)=y} \mu_x(g)=\mu(g)$ (see \cite{Je}, \cite{Je1}).

Let $S$ be the non-properness set of $F$ (see e.g.  \cite{Je}, \cite{Je1}). If $S=\emptyset$, then all mappings $f_p$ are proper. 
Hence we can assume that $S\not= \emptyset.$ Consider the canonical projection $\pi: S\to P.$ We have two possibilities:

\item{(1)} $\pi(S)$ is dense in $P.$

\item{(2)} $\pi(S)$ is not dense in $P.$

In the case a) a generic mapping $f_p$ is not proper. In the second case 
 note that $S$ has dimension dim $P$ + dim $X -1$ and the fiber of $\pi$ has dimension at most dim $X.$ This immediately implies that
 the set $\overline{\pi(S)}$ is a hypersurface in $M$. Moreover, fibers of $\pi$ are the whole space $X.$ This means that 
 for all $p\in \pi(S)$ we have $\mu(f_p)<\mu(F).$ Of course outside $\pi(S)$ mappings $f_p$ are proper. Two such a generic mappings are topologically equivalent by \cite{Je2}, Theorem 4.3.
\end{proof}

Now we state our main result:

\begin{theorem}
Let $X\subset \C^n$ be a smooth irreducible affine variety of dimension $k$ and let $F: X\to \C^m$ be a polynomial mapping.
If $m\ge k$, then there existss a Zariski open dense subset $U$ in the space of linear mappings ${\mathcal L}(\C^n,\C^m)$ 
such that:

a) for every $L\in U$ the mapping $F+L$ is a finite mapping. 

b) all mappings $F+L, L\in U$ are topologically equivalent.
\end{theorem}

\begin{proof}
Let $G: X\ni x \mapsto (x, F(x))\in X\times \C^m$ and $\tilde{X}=graph(G)\cong X.$ Since $m\ge \dim \tilde{X}$  a generic linear projection 
$\pi : \tilde{X}\to \C^m$ is a proper mapping. Hence also the mapping $\pi\circ G$ is proper.
Consequently we get  that for a general matrix $A\in GL(m,m)$ and general linear mapping $L\in {\mathcal L}(\C^n,\C^m)$ the mapping $H(A,L)=A(F_1,...,F_m)^T+L$ is proper. Hence also the mapping $A^{-1}\circ H(A,L)$ is proper. This means that the mapping $F+ A^{-1}(l_1,...,l_m)^T$ (where $L=(l_1,...,l_m)$) is proper. But we can specialize the matrix $A$ to the identity and the mapping $L$ to a given linear mapping $L_0\in {\mathcal L}(\C^n,\C^m)\}$. Hence we see that there is at least dense subset of linear mappings $L\in {\mathcal L}(\C^n,\C^m)$ such that the mapping $F+L: X\to \C^m$ is proper. Consider the algebraic family ${\mathcal F}=\{ F+L, L\in {\mathcal L}(\C^n,\C^m)\}$. By Theorem
\ref{rodzina} we see that there exists a Zariski dense open subset $U\subset {\mathcal L}(\C^n,\C^m)$ such that every mapping $F+L; \ L\in U$ is proper and  all these mappings are topologically equivalent.
\end{proof}

\end{document}